\documentclass[11pt]{amsart}
\usepackage{amsmath,amsthm,amssymb,graphicx,fancyhdr}
\usepackage{natbib}

\newcommand{\prob}{\stackrel{P}{\longrightarrow}}

\newcommand{\one}{{\bf 1}}

\newcommand{\reals}{{\mathbb R}}
\newcommand{\bbr}{\reals}
\newcommand{\vep}{\varepsilon}

\newcommand{\bbc}{\protect{\mathbb C}}

\newcommand{\ams}{\cite{adler:moldavskaya:samorodnitsky:2014}}
\newcommand{\be}{{\bf e}}

\newcommand{\BX}{{\bf X}}
\newcommand{\BY}{{\bf Y}}
\newcommand{\BZ}{{\bf Z}}

\newtheorem{theorem}{Theorem}[section]

\newtheorem{prop}{Proposition}[section]

\newtheorem{lemma}{Lemma}[section]

\newtheorem{ex}{Example}[section]

\newtheorem{remark}{Remark}
\newtheorem{property}{Property}

   \newtheoremstyle{example}{\topsep}{\topsep}%
     {}
     {}
     {\bfseries}
     {}
     {\newline}
     {\thmname{#1}\thmnumber{ #2}\thmnote{ #3}}

   \theoremstyle{example}

\def\Cov{{\rm Cov}}

\def\Var{{\rm Var}}
\def\E{{\rm E}}
\def\P{{\rm P}}

\def\by{{\bf y}}

\numberwithin{equation}{section}

\begin{document}

\title[Asymptotic behaviour of Gaussian minima]
{Asymptotic behaviour of high Gaussian minima}

\author{Arijit Chakrabarty}
\address{Theoretical Statistics and Mathematics Unit \\
Indian Statistical Institute \\
203 B.T. Road\\
Kolkata 700108, India}
\email{arijit.isi@gmail.com}

\author{Gennady Samorodnitsky}
\address{School of Operations Research and Information Engineering\\
and Department of Statistical Science \\
Cornell University \\
Ithaca, NY 14853}
\email{gs18@cornell.edu}

\thanks{Chakrabarty's research was partially supported by the INSPIRE grant of the Department of Science and Technology, Government of India.
Samorodnitsky's research   was partially supported by the ARO
grant  W911NF-12-10385 and by the NSF grant DMS-1506783
at Cornell University.}

\keywords{ Gaussian process, high excursions, minima, precise asymptotics
\vspace{.5ex}}

\begin{abstract}
We  investigate what happens when an entire sample path of a smooth
Gaussian process on a compact interval lies above a high
level. Specifically, we determine the precise asymptotic probability
of such an event, the extent to which the high level is exceeded, 
the conditional shape of the process above the high level, and the
location of the minimum of the process given that the sample path is
above a high level.  
\end{abstract}

\maketitle

\section{Introduction}
\label{sec:intro}

Extremal behaviour of Gaussian processes has been the subject of
numerous studies. It is of interest from the point of view of the
extreme value theory, or large deviations theory and of the theory of
sample path properties of stochastic processes. The asymptotic
distribution of 
the supremum of bounded Gaussian processes has been very thoroughly
studied; highlights include \cite{dudley:1973},
\cite{berman:kono:1989} and \cite{talagrand:1987}, and the books of
\cite{piterbarg:1996}, \cite{adler:taylor:2007} and
\cite{azais:wschebor:2009}. In this paper we are interested in another
type of the asymptotic behaviour of Gaussian processes: the situation
when an entire sample path of the process is above a high level. Such
situations are important for understanding the structure of the high
level excursion sets of Gaussian random processes and fields. 

Very loosely speaking,  we are interested in the asymptotics of the
Gaussian minima when these minima are high. Dealing with high Gaussian
minima is not easy. A finite-dimensional situation (in the language of
dependent lognormal random variables) is considered in
\cite{guliashvili:tankov:2016}. We, on the other hand, consider minima
of zero mean sample continuous Gaussian
 processes. The processes we consider are often stationary, but some
 nonstationary processes fall within our framework as well. 

We now describe the questions of interest to us more concretely. 
Let $\BX:=(X_t:t\in\bbr)$ be a centered Gaussian
 process with continuous paths, defined on  some probability
 space $(\Omega,{\mathcal F},\P)$. 
Let $[a,b]$ be a compact interval, and let $u>0$ be a high level.  We
study a number of problems related to the situation described above,
i.e. the situation when the entire 
sample path $(X_t:t\in [a,b])$ lies above the level $u$. Specifically,
we are interested in the following questions. 

{\bf Question 1}.  What is the precise asymptotic behaviour of the
probability 
\[
\P\left(\min_{a\le t\le b}X_t>u\right)
\]
as $u\to\infty$ ?

{\bf Question 2}.    Given the event
\begin{equation} \label{e:Bu}
B_u:=\left\{\min_{a\le t\le b}X_t>u\right\}\,,
\end{equation}
how does the conditional distribution of $(X_t:t\in [a,b])$ behave as $u\to\infty$ ?

{\bf Question 3}.  Conditionally on $B_u$, what can be said about the
asymptotics of the overshoot 
\[
\min_{a\le t\le b}X_t-u\,,
\]
as $u\to\infty$ ?

{\bf Question 4}.  Consider the location of the minimum of the
process, 
\[
\arg\min_{a\le t\le b}X_t
\]
taken to be the leftmost location of the minimum in case there are
ties (it is elementary that this location is a well defined random
variable). What is the asymptotic distribution of the location of the
minimum given $B_u$, as $u\to\infty$ ?

Some information on Questions 1 and  2 is contained in
\ams~. Regarding {\bf Question 1}, the latter 
  paper describes the
probabilities of the type $\P\left(\min_{a\le t\le b}X_t>u\right)$ 
   on the logarithmic level, while in the present paper we are 
interested in precise asymptotics of that probability. Regarding {\bf
  Question 2}, 
the latter paper studies the asymptotic behaviour of the ratio 
$$
\frac1u X_t, \, a\le t\le b
$$
given $B_u$, as $u\to\infty$, while in the present paper we would like
to know the deviations of the sample path from this linear in $u$
behaviour.  Furthermore, the paper of \ams\ provides no information on
{\bf Question 3} and {\bf Question 4} above.  

For stationary (not necessarily Gaussian) processes a general theory
of the location of the supremum (or infimum) of the process is
developed in \cite{samorodnitsky:shen:2013}. However, the limiting
behaviour of the minimum location in {\bf Question 4} even in the
stationary case is outside of that theory.

We obtain fairly precise answers to the above questions. However, in
order to achieve this level of precision, we will impose much stricter
smoothness assumptions on the process $\BX$ then those imposed in
\ams. We describe the precise assumptions on the process in Section
\ref{sec:smoothness}. Section \ref{sec:preliminary} contains
preliminary results, while the main results of the paper with answers
to Questions 1-4 are stated in Section \ref{sec:results}. Section
\ref{sec:examples} presents two examples illustrating the main results
of the paper. The results stated in Section \ref{sec:results} are proved in Section \ref{sec:proofs}.

\section{Assumptions on the process $\BX$}
\label{sec:smoothness}

In this section we will state and discuss the assumptions
on the Gaussian process $\BX$ we will use in the rest of the
paper. Among others, these assumptions will guarantee that our process
is very smooth. 
Our main interest lies in stationary Gaussian processes, and
for these processes the assumptions are easy to
state. However, our main results in the subsequent sections do not
depend on the stationarity of the process. Rather, they depend on
certain properties of the process which follow, in the stationary case,
from a small number of basic assumptions. These properties are
discussed in the remainder of this section. 

We use the notation 
\[
R(s,t):=\E(X_sX_t)\,,s,t\in\bbr
\]
for the  covariance function of the process $\BX$. If the process is
stationary, then its covariance function is related to the spectral
measure of the process by writing (with the usual abuse of notation
related to the dual use of $R$ to denote both a function of one
variable and a function of two variables)
\[
R(s,t) = R(t-s) =\int_{-\infty}^\infty e^{\iota(t-s)x}F_X(dx)\,,s,t\in\bbr\,,
\]
where $\iota:=\sqrt{-1}$.
Recall that the spectral measure $F_X$ of the process $\BX$ is a
finite symmetric Borel measure on $\bbr$. 

When the process $\BX$ is stationary, we will impose the following
conditions on the spectral measure $F_X$. 
\begin{itemize}
\item[S1.] For all $t\in\bbr$,
\begin{equation}
\label{p0.eq1}
\int_{-\infty}^\infty e^{tx}F_X(dx)<\infty\,.
\end{equation}
\item[S2.] The support of $F_X$  has at least one accumulation point.
\end{itemize}

The canonical example of such a Gaussian process is the process with
the Gaussian spectral density
\begin{equation} \label{e:Gauss.sp.dens}
F(dx) = \frac1{\sqrt{2\pi}}e^{-x^2/2}\, dx, \ x\in\bbr\,.
\end{equation}
This process was considered in detail in \ams, and we will use it in
this paper to illustrate our results. 

The following proposition establishes certain consequences of the
conditions S1 and S2 in the case of a stationary process. It is these
consequences, rather than stationarity itself, that will be used in
much of the paper.
\begin{prop}\label{p0}
Let $\BX$ be a stationary Gaussian process whose spectral measure
satisfies S1 and S2. Then the process $\BX$ has the following
properties. 
\begin{property}   \label{assume1} The function $R(\cdot,\cdot)$ has a
  power expansion 
\[
R(s,t)=\sum_{m=0}^\infty\sum_{n=0}^\infty r_{mn}s^mt^n\,,s,t\in\bbr\,,
\]
for some $(r_{mn}:m,n\ge0)\subset\bbr$.
\end{property}

\begin{property}     \label{assume2} 
For any compact interval $[a,b]$, the family $(X_t:t\in[a,b])$ is non-negatively
  non-degenerate. That is, for any probability measure $\nu$ on
  $[a,b]$, 
\[
\Var\left(\int_a^bX_t\nu(dt)\right)=\int_a^b\int_a^b R(s,t)\nu(ds)\nu(dt)>0\,.
\]
\end{property}

\begin{property}    \label{assume3} 
The sample paths of $\BX$ are infinitely differentiable, and 
  the covariance matrix of any finite
  sub-collection  of the family
  $(X_t^{(n)}:t\in\bbr,\,n=0,1,\ldots)$  is non-singular. Here and
  elsewhere, for any function 
  $f$ and $n\ge0$, $f^{(n)}$ denotes its $n$-th derivative whenever it
  exists, with $f^{(0)}=f$. 
\end{property}
\end{prop}

\begin{proof}
The integral
$$
\tilde R(z):=\int_\bbr e^{\iota zx}\, F_X(dx)\,,z\in\bbc\,,
$$
defines, clearly, an analytic function, and then $R(s,t)=\tilde
R(t-s)$ for $s,t\in\bbr$ has Property  \ref{assume1}. 

To check Property  \ref{assume2}, suppose for the sake of
contradiction, that there exists a probability measure $\nu$ on 
$[a,b]$ such that 
\[
\int_a^b\int_a^b R(s,t)\nu(ds)\nu(dt)=0\,.
\]
Clearly, the left hand side is the same as
\[
\int_{-\infty}^\infty\left|\int_a^b e^{\iota tx}\nu(dt)\right|^2F_X(dx)\,.
\]
Therefore, it follows that
\[
\int_a^b e^{\iota tx}\nu(dt)=0\,,
\]
for every $x$ in the support of $F_X$. The left hand side is an
analytic function of the complex variable $x$. Since the equality
holds on the support of $F_X$, which has an accumulation point, by the
assumption S2, the equality holds for all $x\in\bbc$ including
$x=0$. This contradicts the fact that $\nu$ is a probability
measure and thus verifies Property \ref{assume2}.

For Property \ref{assume3} notice that the integral 
\begin{equation}
\label{p0.eq2}Y(z):=\int_\bbr e^{\iota zx}Z(dx)\,,z\in\bbc\,,
\end{equation}
where $Z$ is a complex-valued Gaussian random measure with control
measure $F_X$, has a version that is a random analytic function. Indeed, by \eqref{p0.eq1} all the integrals of the type
$$
\int_\bbr x e^{\iota zx}Z(dx)
$$
are well defined complex-valued Gaussian random variables for any $z\in\bbc$. 
This gives a version of the random function \eqref{p0.eq2} that satisfies the Cauchy-Riemann conditions. Since the  restriction of this random function to 
$z\in\bbr$ coincides distributionally with $\BX$, there is a version
of $\BX$ whose sample paths   are infinitely differentiable. Now
Property \ref{assume3} follows from Exercise 3.5 in
\cite{azais:wschebor:2009}. 

This completes the proof of the proposition.
\end{proof}

\section{Preliminary results}
\label{sec:preliminary}

We fix an interval $[a,b]$, once and for all, where $-\infty<a<b<\infty$, and proceed with a number of preliminary results that are important for
the  main results in the subsequent sections. Most of these results
address several issues related to the optimization problem 
\begin{equation}
\label{p1.eq1}
\min_{\nu\in M_1[a,b]}\int_a^b\int_a^bR(s,t)\nu(ds)\nu(dt)\,,
\end{equation}
where the minimum is taken 
over all Borel probability measures on $[a,b]$. The
importance of this problem to the questions studied in this paper was
shown in \ams. 

\begin{prop}
\label{p1} Let $\BX$ be a Gaussian process satisfying either (a) or (b) below.

\begin{enumerate}

\item[(a)] The process has Property 1, Property 2 and
Property 3, and additionally,  for every fixed $s\in [a,b]$,
\begin{equation}  \label{assume4} 
\limsup_{t\to\infty}R(s,t)\le0\,.
\end{equation}

\item[(b)] The process  is  stationary whose spectral measure
satisfies S1 and S2.

\end{enumerate}

Then the minimization problem \eqref{p1.eq1} has a unique
minimizer $\nu_*$. Furthermore, the support of  
$\nu_*$ has a finite cardinality and the minimum value in \eqref{p1.eq1}
is strictly positive. 
\end{prop}

\begin{proof} 
We start with the proof under the assumption (a). The fact that the minimum is achieved follows from continuity of the functional being optimized in the topology of
weak convergence on $M_1[a,b]$ and compactness of that space. 
The claim that the minimum value  is positive follows from Property
\ref{assume2}. Let 
$\nu_*$ be a minimizer in the optimization problem \eqref{p1.eq1}, 
and define 
\begin{equation} \label{e:Y}
Y:=\int_a^bX_s\, \nu_*(ds)\,.
\end{equation} 
Let 
\[
\hat\mu(t):=\E\left(X_t|Y=1\right),\, t\in\bbr\,,
\]
and observe that
\begin{equation}
\label{eq.defmuhat}\hat\mu(t)=\frac{\Cov(X_t,Y)}{\Var(Y)}=
\frac{\int_a^bR(s,t)\, \nu_*(ds)}{\Var(Y)}, \ t\in\bbr\,. 
\end{equation}

By Theorem 5.1 in \ams, $\hat\mu\ge1$ on $[a,b]$, and $\hat\mu=1$ on
the support of $\nu_*$. Since $\hat\mu$ is real analytic on $\bbr$ by
Property \ref{assume1}, there are two possibilities. Either 
\begin{equation}
\label{p1.eq2}\hat\mu(t)=1\text{ for all }t\in\bbr\,,
\end{equation}
or the set
\begin{equation}\label{eq.defhats}
\hat S:=\{t\in[a,b]:\hat\mu(t)=1\}
\end{equation}
has no accumulation points and, hence, is a set of finite cardinality. In
the latter case, the support of $\nu_*$ is also a finite set. We will
show that \eqref{p1.eq2} is impossible and, hence, the latter option is
the only possible one. 

Indeed, suppose that \eqref{p1.eq2} holds. Then the function
\[
g(t):=\int_a^bR(s,t)\, \nu_*(ds)\,,t\in\bbr
\]
is a positive constant. However, by \eqref{assume4} and
Fatou's lemma, it follows that 
\[
\limsup_{t\to\infty}g(t)\le0\,,
\]
leading to a contradiction. We conclude that $\hat S$ is a finite set
and so is the support of $\nu_*$.  

In order to prove the uniqueness of an optimal measure, suppose that
$\nu_1$ and $\nu_2$ are two different optimal measures. By Property 
\ref{assume3}, the finitely many random variables $X_t$ for $t$
in the union of the 
supports of the two measures are linearly independent and, hence, the 
function 
\[
\alpha\mapsto\Var\left(\int_a^bX_s\bigl(\alpha\nu_1(ds)+(1-\alpha)\nu_2(ds)\bigr)\right)
\]
is strictly convex on $[0,1]$. Such a function cannot take the same
minimal value at the two endpoints $0$ and $1$, and the uniqueness
follows. 

We now prove the same claim under the assumption (b). By Proposition \ref{p0} the
assumptions S1 and S2 on the spectral measure of a stationary Gaussian
process $\BX$ imply Property \ref{assume1}, Property \ref{assume2} and
Property \ref{assume3}, so the only 
ingredient missing in an attempt to apply the statement of part (a) to
part (b) is that, in part (b), we have not assumed
\eqref{assume4}. Since the only place in the proof of part (a) where
\eqref{assume4} is used, is in ruling out  \eqref{p1.eq2}, we only
need to show that  \eqref{p1.eq2} can be ruled out under the
assumptions of part (b) as well. 

Indeed, suppose that \eqref{p1.eq2} holds. The assumption S1 implies that
 $R$ can be extended to an analytic  function on $\bbc\times\bbc$ by
 \[
 R(s,t)=\int_{-\infty}^\infty e^{\iota(t-s)x}F_X(dx), \, s,t\in\bbc \,.
 \]
Then the analytic 
function 
$$
g(t) = \int_a^b R(s,t)\, \nu(ds), \ t\in\bbc
$$
must be a real constant. Note that 
$$
g(t) = \int_{-\infty}^\infty e^{\iota tx} h(x)\, F_X(dx)\,,
$$
where
$$
h(x) = \int_a^b e^{-\iota sx}\, \nu(ds), \ \ x\in\bbr\,.
$$
Write
$$
h(x) = h_1(x) +\iota h_2(x), \ x\in\bbr\,,
$$
where $h_1$ is a real even function and $h_2$ is a real odd
function. Then, since $g$ is a real constant, we have
$$
g(t) = \int_{-\infty}^\infty \cos tx\,  h_1(x)\, F_X(dx) -
\int_{-\infty}^\infty \sin tx\,  h_2(x)\, F_X(dx), \ t\in\bbr\,.
$$
Since $g$ is an even function (a constant one), the second term in the
right hand side vanishes, so that
$$
g(t) = \int_{-\infty}^\infty \cos tx\,  h_1(x)\, F_X(dx), \ t\in\bbr\,.
$$
By the uniqueness of the Fourier transform, the only finite signed
measures to have constant transforms are point masses at the origin,
so we must have $h_1=0$ $F_X$-a.e. on $\{x\not=0\}$. Since the
function $h_1$ is real analytic, and the support of $F_X$ has an
accumulation point, we conclude that $h_1$ = 0 everywhere. This is not
possible since $h_1$ is the characteristic function of the probability
measure obtained by making $\nu$ a symmetric probability measure on
$[a,b]\cup[-b,-a]$. This rules out \eqref{p1.eq2}.  
\end{proof}

\begin{remark}
\label{rk1}
{\rm
Interestingly, in the non-stationary case  the statement of
Proposition \ref{p1} might be false if \eqref{assume4} is not
assumed. To see this,  consider the 
following example. Let $Y$ be a standard normal random
variable and $(Z_t:t\in\bbr)$ be a stationary mean zero Gaussian
process, independent of $Y$, with covariance 
\[
\E(Z_sZ_t)=e^{-(s-t)^2/2}\,,s,t\in\bbr\,.
\]
Define
\[
X_t=Y+Z_t-\int_0^1Z_sds\,,t\in\bbr\,.
\]
Clearly, the process $(X_t:t\in\bbr)$ has Property \ref{assume1},
Property \ref{assume2} and Property \ref{assume3}. 
However, if $[a,b]=[0,1]$, then 
the minimizer in \eqref{p1.eq1} is the  Lebesgue measure on $[0,1]$,
and it does not have a support of a finite cardinality. 
}
\end{remark}

We denote by $S$ the support of the unique minimizer $\nu_*$ in the
minimization problem \eqref{p1.eq1}. Two objects related to this set
will be of crucial importance in the sequel. First of all, we let 
\begin{equation}\label{eq.defmu}
\mu(t):=\E\left(X_t|X_s=1\text{ for all }s\in S\right)\,,t\in\bbr\,.
\end{equation}
The importance of the function $\mu$ stems from the following claim:
conditionally on the event $B_u$ in \eqref{e:Bu}, as $u\to\infty$, 
\begin{equation} \label{e:mu.xab}
\bigl( u^{-1}X(t), \, a\leq t\leq b\bigr) \to \bigl(\mu(t),  \, a\leq
t\leq b\bigr)
\end{equation} 
in probability, in $C[a,b]$. Indeed, it was shown in \ams\ that
\eqref{e:mu.xab} holds with the function $\mu$ replaced by the
function $\hat \mu(t)=\E\left(X_t|Y=1\right)$, $t\in\bbr$, 
defined  in the proof of Proposition 
\ref{p1}. Recall that $Y=\int_a^bX_s\, \nu_*(ds)$. 
Therefore,  we only need to show that
$\mu=\hat\mu$. Enumerate the elements of $S$ as 
$$
S=\{t_1,\ldots,t_k\}
$$ 
(this is the notation we will use throughout the paper) and write
$$
\E\bigl( X_t|
X_{t_1},\ldots, X_{t_k}\bigr) = \sum_{j=1}^k a_j(t)X_{t_j},\, t\in\bbr\,.
$$
Then 
\begin{align*}
\hat\mu(t)Y & = \E\bigl( X_t|Y\bigr) = \E\Bigl( \E\bigl( X_{t}|
X_{t_1},\ldots, X_{t_k}\bigr)\big| Y\Bigr) \\
& = \E\left( \sum_{j=1}^k a_j(t)X_{t_j} \Big| Y\right) = \sum_{j=1}^k
  a_j(t) \hat\mu(t_j) Y\\
& = \sum_{j=1}^k   a_j(t)  Y
\end{align*}
since $\hat\mu$ is equal to one at all points of the support of the
measure $\nu_*$. That is,
$$
\hat\mu(t) = \sum_{j=1}^k   a_j(t)  = \mu(t)\,,
$$
as required. 

We record for future use several useful facts about the function
$\mu$.
\begin{lemma} \label{l:mu.prop}
The function $\mu$ is a restriction to $\bbr$ of an analytic function on $\bbc$
and for each $j=0,1,2,\ldots$ and $t\in\bbr$, 
\begin{equation} \label{e:der.condE} 
\E\left(X^{(j)}_t|X_s=1\text{ for all }s\in S\right)= \mu^{(j)}(t), \
t\in\bbr\,.
\end{equation}
Further, for each $s\in S\cap (a,b)$ there exists an even positive integer $n$ such that
\begin{equation*}
\mu^{(1)}(s)=\ldots=\mu^{(n-1)}(s)=0<\mu^{(n)}(s)\,.
\end{equation*}
In particular, there is $\vep>0$ such that 
$\mu^{(2)}(t)>0$ for each $t\in (s-\vep,s)\cup (s,s+\vep)$. 

Similarly, if $a\in S$, then $\mu^{(1)}(a)\ge0$. If equality holds, then
 there is $\vep>0$ such that 
\[
\mu^{(2)}(t)>0,\,\text{for each }t\in   (a, a+\vep)\,.
\]
If  $b\in S$, then $\mu^{(1)}(b)\le0$. If equality holds, then
there is $\vep>0$ such that  
\[
\mu^{(2)}(t)>0,\,\text{for each }t\in   (b-\vep,b)\,.
\]
\end{lemma}

\begin{proof}
We already know that $\mu$ is a restriction to $\bbr$ of an analytic
function, and \eqref{e:der.condE} is obvious. The final property of
$\mu$ follows from the fact that it is analytic. 
\end{proof}

If we  define the ``essential'' set by 
\begin{equation} \label{e:ess.set}
E:=\{t\in[a,b]:\mu(t)=1\}\,,
\end{equation}
then we have proved above that $E=\hat S$,  defined in
\eqref{eq.defhats}. In particular, we showed in the proof of
Proposition \ref{p1} that under the assumptions of the proposition,
the essential set $E\supseteq S$ is a finite set as well. 
 
It turns out that in many cases the support $S$ of the optimal measure
for the optimization problem \eqref{p1.eq1} contains the endpoint of
the interval. This holds, in particular, under certain monotonicity
assumption in the covariance function of the process. Specific 
sufficient conditions are given in the following proposition. 
\begin{prop}\label{p2}
(a) \ Suppose that the 
following two conditions hold: 
 for all $a\le s_1\le s_2\le s_3\le b$, 
\begin{equation}\label{p2.eq1}
R(s_1,s_3)\le\min\{R(s_1,s_2),R(s_2,s_3)\}\,,
\end{equation}
and for all $s,t\in [a,b]$, 
\begin{equation}
\label{p2.eq2}R(s,t)<\min\{R(s,s),R(t,t)\}\text{ whenever }s\neq t\,. 
\end{equation}
Then any finite support optimal measure for the optimization problem \eqref{p1.eq1}
puts positive masses at the endpoints $a$ and $b$ of the interval. 

(b) \ Suppose that $\BX$ is a stationary Gaussian process whose 
covariance function $R$ is nonincreasing on $[0,b-a]$. If the spectral measure $F_X$
is not a point mass at the origin, then the conclusion of part (a)
holds. 
\end{prop}
\begin{proof}
We start with part (a). 
Assume, to the contrary, that $\nu$ puts no mass at the
point $a$.  Enumerate the elements of the support of $\nu$ (which we
still denote by $S$)  as $\{t_1,\ldots,t_k\}$, with the smallest of
these elements, $t_1>a$. If $\nu(\{ t_i\})=\alpha_i>0, \, i=1,\ldots,
k$, then 
\[
\Var\left(\int_a^bX_s\, \nu(ds)\right)=\sum_{i=1}^k\sum_{j=1}^k\alpha_i\alpha_jR(t_i,t_j)\,.
\]
Define a probability measure
\[
\nu_\vep:=\vep\delta_{a}+(\alpha_1-\vep)\delta_{t_1}+\sum_{i=2}^k\alpha_i\delta_{t_i}\,,0\le\vep\le\alpha_1\,.
\]
Notice that $\nu_0\equiv\nu$, and
\begin{eqnarray*}
\frac d{d\vep}\Var\left.\left(\int_a^bX_s\, \nu_\vep(ds)\right)\right|_{\vep=0}&=&2\sum_{i=1}^k\alpha_i\left[R(a,t_i)-R(t_1,t_i)\right]\\
&<&0\,.
\end{eqnarray*}
The inequality follows from  the observation that  $a<t_1\le t_i$ for
all $i$ and hence by \eqref{p2.eq1} the summands are non-positive, and
the term with $i=1$ is strictly negative by \eqref{p2.eq2}. This
contradicts the fact that $\nu$ is an optimal measure. Thus,
$a\in S$. A similar argument shows that $b\in S$. 

For part (b), we only need to check that the assumptions of part (a)
hold. The assumption \eqref{p2.eq1} follows from monotonicity of the
covariance function of the stationary process. The only additional
argument needed for \eqref{p2.eq2} is the observation that, unless the
spectral measure is concentrated at the origin, the covariance
function cannot be constant in an neighborhood of the origin. 
\end{proof}
 
We will use in the sequel several facts about the finite-dimensional
centered Gaussian vector $(X_t,t\in S)$. These facts are collected in
the proposition below. We will use the common notation $f(u)\sim
c\,g(u)$ as $u\to\infty$ (where $g$ is a non-vanishing function) to
mean that
\begin{equation}
\label{eq.defsim}\lim_{u\to\infty}\frac{f(u)}{g(u)}=c\,.
\end{equation}
Note that the possibility  $c=0$ is allowed.

In the sequel, all vectors are column vectors unless mentioned
otherwise. We use the notation 
$\one$ for a vector with entries equal
to $1$, whose dimension  is clear from the context. 

\begin{prop} \label{l:Gauss.vector}
Under the assumptions of either part (a) or part (b) of Proposition
\ref{p1}, let $S:=\{t_1,\ldots,t_k\}$
 be the  finite cardinality support of the unique minimizer
in \eqref{p1.eq1}, and let $\Sigma$ be the covariance matrix of the
vector $\bigl( X_{t_1}, \ldots, X_{t_k}\bigr)$. 

(i) Denote 
\begin{equation}
\label{e:deftheta}\theta:=\Sigma^{-1}\one\,,
\end{equation}
Then,
\begin{equation}
\label{t2.claim1}\theta_j>0\text{ for all }1\le j\le k\,.
\end{equation}

(ii)  Conditionally on the event $\{ \min_{t\in S} X_t>u\}$, we have
\begin{equation}
\label{l:Gauss.vector:claim}\bigl(u(X_{t_1}-u),\ldots,u(X_{t_k}-u)\bigr)\Rightarrow\bigl(E_1,\ldots,E_k\bigr)
\end{equation}
as $u\to\infty$, where
$(E_1,\ldots,E_k)$ are independent exponential random variables with
parameters $\theta_1,\ldots, \theta_k$ respectively. 

(iii) The distributional tail of the minimal component of the Gaussian
vector $\bigl( X_{t_1}, \ldots, X_{t_k}\bigr)$ satisfies 
\begin{equation}
\label{e:equiv2}\P\bigl( \min_{t\in S} X_t>u\bigr)\sim
\frac{c}{\theta_1\ldots \theta_k} u^{-k}
\exp\left\{ -\frac12 u^2(\theta_1+\ldots +\theta_k)\right\}
\end{equation}
as $u\to\infty$, where 
\begin{equation} \label{e:c}
c=(2\pi)^{-k/2}(\det\Sigma)^{-1/2}\,. 
\end{equation} 
\end{prop}

In \eqref{l:Gauss.vector:claim} and  similar statements in the sequel, the
law of the random vector in the left hand side is computed, for every
$u>0$, as the conditional law given $\min_{t\in S} X_t>u$. That is, \eqref{l:Gauss.vector:claim} means that
\[
P\left(\left.\bigl(u(X_{t_1}-u),\ldots,u(X_{t_k}-u)\bigr)\in\cdot\right|\min_{t\in S} X_t>u\right)\Rightarrow P\left((E_1,\ldots,E_k)\in\cdot\right)\,,
\]
weakly in $\bbr^k$, as $u\to\infty$.

\begin{proof}[Proof of Proposition \ref{l:Gauss.vector}]
We start with part (i). Recall that by Property 3, the inverse matrix 
 $\Sigma^{-1}=(\sigma^{-1}_{ij})$ is well defined and, hence, so is 
the vector $\theta$. 
Suppose, for instance, that, to the contrary,  
\begin{equation} \label{e:neg.sum}
\sum_{j=1}^{k} \sigma^{-1}_{1j}\leq 0\,.
\end{equation}
Define a $k\times1$ vector $\lambda^{(0)}$ by
\[
\lambda^{(0)}_j:=\nu(\{t_j\})\,,1\le j\le k\,,
\]
where $\nu$ is the unique  minimizer in \eqref{p1.eq1}. In other
words, $\lambda^{(0)}$ is the minimizer in the problem 
\begin{equation}
\label{e:minimize}\min_{\lambda\in\bbr^k_+:\, \sum_{i=1}^k\lambda_i=1}\lambda^T\Sigma\lambda\,.
\end{equation}  
Clearly,  $\lambda_i^{(0)}>0$ for all $i=1,\ldots,k$. For
small $\vep>0$ consider vectors of the form
$$
\lambda^{(\vep)} =\lambda^{(0)} -\vep\Sigma^{-1} \be^{(1)}\,,
$$
where $\be^{(1)}=(1,0,\ldots,0)$. Then it follows from
\eqref{e:neg.sum} that 
\begin{equation} \label{e:atleast.1}
\sum_{i=1}^n \lambda_i^{(\vep)}\geq 1. 
\end{equation}
Note that
$$
(\lambda^{(\vep)})^T\Sigma \lambda^{(\vep)} = (\lambda^{(0)})^T\Sigma
\lambda^{(0)} -2\vep \lambda^{(0)}_1 + \vep^2 (\be^{(1)})^T\Sigma^{-1}
\be^{(1)}\,.
$$
Since $\lambda^{(0)}_1>0$, for small $\vep>0$, this expression is
strictly smaller than 
$$
(\lambda^{(0)})^T\Sigma \lambda^{0} \,.
$$
Additionally, for small $\vep>0$ the vector $\lambda^{(\vep)}$ has
positive components. Recalling 
\eqref{e:atleast.1}, this contradicts the optimality of
$\lambda^{(0)}$ in \eqref{e:minimize}. This contradiction shows that
$\theta_1>0$. Similarly, $\theta_j>0$ for all $1\leq j\leq k$. Hence,
\eqref{t2.claim1} holds. 

For parts (ii) and (iii)  we start by noticing that we 
can write, with $c>0$ given by \eqref{e:c},  for any $h_i\geq 0, \,
i=1,\ldots, k$, for $u>0$, 
\begin{align*}
& \P\bigl( X_{t_i}>u+h_i/u, \ i=1,\ldots, k\bigr) \\
=&c\int_{u+h_1/u}^\infty \ldots \int_{u+h_k/u}^\infty
\exp\left\{ -\frac12 \by^T \Sigma^{-1}\by\right\}\, dy_1\ldots dy_k \\
=&c\int_{h_1/u}^\infty \ldots \int_{h_k/u}^\infty
\exp\left\{ -\frac12 \sum_{i=1}^k\sum_{j=1}^k (y_i+u)(y_j+u)\sigma^{-1}_{ij}
\right\}\, dy_1\ldots dy_k \\
=&c \exp\left\{ -\frac12 u^2\sum_{i=1}^k\sum_{j=1}^k
  \sigma^{-1}_{ij}\right\} \int_{h_1/u}^\infty \ldots\\
& \ldots \int_{h_k/u}^\infty
\exp\left\{ -\frac12 \sum_{i=1}^k\sum_{j=1}^k
  y_iy_j\sigma^{-1}_{ij}\right\} 
\exp\left\{ -u\sum_{i=1}^ky_i\sum_{j=1}^k   \sigma^{-1}_{ij}\right\} 
\, dy_1\ldots dy_k \\
=&c \exp\left\{ -\frac12 u^2(\theta_1+\ldots +\theta_k)
\right\} \\
&\int_{h_1/u}^\infty \ldots \int_{h_k/u}^\infty
\exp\left\{ -\frac12 \sum_{i=1}^k\sum_{j=1}^k
  y_iy_j\sigma^{-1}_{ij}\right\} \exp\left\{
  -u\sum_{i=1}^k\theta_i y_i \right\}  
\, dy_1\ldots dy_k \\
:=&c \exp\left\{ -\frac12 u^2(\theta_1+\ldots +\theta_k)
\right\} \int_{h_1/u}^\infty \ldots \int_{h_k/u}^\infty A(y_1,\ldots, y_k)
\, dy_1\ldots dy_k \,.
\end{align*}
For $\vep>0$ we write
\begin{align*}
&\int_{h_1/u}^\infty \ldots \int_{h_k/u}^\infty A(y_1,\ldots, y_k)
\, dy_1\ldots dy_k \\
= & \int_{h_1/u}^\infty \ldots \int_{h_k/u}^\infty \one\bigl(
\max(y_1,\ldots, y_k)>\vep\bigr) A(y_1,\ldots, y_k) \, dy_1\ldots dy_k 
\\
+ &\int_{h_1/u}^\vep \ldots \int_{h_k/u}^\vep  A(y_1,\ldots, y_k)
\, dy_1\ldots dy_k := I_{\vep, 1}(u) + I_{\vep,2}(u)\,.
\end{align*}
Note that
$$
c  I_{\vep, 1}(u) \leq \exp\bigl\{ -\vep u\min_{i=1,\ldots,
  k}\theta_i\bigr\} \P(X_{t_i}>u,\ i=1,\ldots, k)\,.
$$
On the other hand,
\begin{align*}
& I_{\vep,2}(u) \left( \int_{h_1/u}^\infty \ldots \int_{h_k/u}^\infty
\exp\left\{ - u\sum_{i=1}^k \theta_i y_i \right\}  \, dy_1\ldots
dy_k\right)^{-1} \\
&\in \left( \exp\left\{ -\frac12\vep^2\sum_{i=1}^k\sum_{j=1}^k \big|
    \sigma^{-1}_{ij}\big|\right\},1\right)\,.
\end{align*}
Since
\begin{align*}
&\int_{h_1/u}^\infty \ldots \int_{h_k/u}^\infty
\exp\left\{ - u\sum_{i=1}^k \theta_i y_i \right\}  \, dy_1\ldots
dy_k \\
=& u^{-k}\bigl( \theta_1\ldots \theta_k\bigr)^{-1} \exp\bigl\{
-\sum_{i=1}^k \theta_ih_i\bigr\}\,,
\end{align*}
we conclude that as $u\to\infty$,
$$
\int_{h_1/u}^\infty \ldots \int_{h_k/u}^\infty A(y_1,\ldots, y_k)
\, dy_1\ldots dy_k
  \sim u^{-k}\bigl( \theta_1\ldots \theta_k\bigr)^{-1}
\exp\bigl\{ -\sum_{i=1}^k \theta_ih_i\bigr\}
$$
and, hence, 
\begin{eqnarray}
\label{e:equiv1}&& \P\bigl( X_{t_i}>u+h_i/u, \ i=1,\ldots, k\bigr) \\
\nonumber&\sim&\frac{c}{\theta_1\ldots\theta_k}\exp\left\{ -\sum_{i=1}^k \theta_ih_i\right\}u^{-k}\exp\left\{ -\frac12 u^2(\theta_1+\ldots +\theta_k)\right\} 
\end{eqnarray}
as $u\to\infty$. Putting $h_1=\ldots=h_k=0$ we obtain
\eqref{e:equiv2}, which together with \eqref{e:equiv1} proves the 
claim of part (ii). 
\end{proof}

\section{The main results}\label{sec:results} 

In this section we will answer Questions 1-4 mentioned in the
introduction. These questions are all centered around the overall
infimum $\min_{t\in[a,b]}X_t$ of the process $\BX$ and the behaviour
of the process when the overall infimum is large. It turns out that,
when the overall infimum is large, its behaviour is similar, but not identical,
to the behaviour of a simpler object - the minimal value in a finite-dimensional
Gaussian vector, formed by several key observations of the process. 
Understanding what happens when the latter minimum is large is an
important ingredient in our analysis in this section. 

Recall (by Proposition \ref{p1}) 
 that under the assumptions we are imposing in this paper, the
optimization problem \eqref{p1.eq1} has a unique minimizer, a
probability measure with a support of a finite cardinality, which we
denote by $S$. Then the minimal value in a finite-dimensional
Gaussian vector mentioned above is simply $\min_{t\in S}X_t$. 

We start with Question 1 of the introduction. We will need additional
notation, which we introduce now. Let, once again, $S$ be the support
of a finite cardinality of the unique minimizer in
\eqref{p1.eq1}. Denote 
\begin{eqnarray}
\label{eq.defy}Y_t&:=&\E(X_t|X_s, s\in S)\,,\\
\label{eq.defz}Z_t&:=&X_t-Y_t\,,
\end{eqnarray}
$t\in\bbr$. Then $(Y_t:t\in\bbr)$ and $(Z_t:t\in\bbr)$ are two
centered Gaussian processes. Moreover, 
the process $(Z_t:t\in\bbr)$ is independent of $(X_s,s\in S)$. In
particular, the process $(Z_t:t\in\bbr)$ is independent of the random
variable $\min_{t\in S}X_t$. Recall that under the assumptions of
Proposition \ref{p1} (which we will always assume), the sample paths of the processes
$\BY$ and $\BZ$ are in $C^\infty$. 

Recall the definition of the essential set $E$ in
\eqref{e:ess.set}, which is a (not necessarily strict) superset of $S$. It is a finite set, and we will enumerate its
points as $E=\{t_1,\ldots,t_l\}$, in such a way that the first $k$
points form the support of the unique minimizer in
\eqref{p1.eq1},  i.e. $S=\{t_1,\ldots,t_k\}$, $k\leq l$, and $a\leq
t_1<\ldots <t_k\leq b$. 

Let $\theta$ be the $k$-dimensional vector with positive coordinates
defined in \eqref{e:deftheta}, and let $\mu$ be the function defined
in \eqref{eq.defmu}. We define several random variables. Let 
$$
W_{(a,b)} = 
\exp\left\{ - \frac12\sum_{j=1,\ldots, k:\, t_j\in  (a,b)} 
  \frac{\theta_j}{\mu^{(2)}(t_j)}\bigl( Z_{t_j}^{(1)}\bigr)^2\right\}\,,
$$
following the usual convention that a positive number divided by zero is plus infinity, and $e^{-\infty}=0$. That is, the right hand side of the above is to be interpreted as zero if $\mu^{(2)}(t_j)=0$ for some $j$ such that $t_j\in(a,b)$.
Let, further,  
$$
W_a= 
\one\bigl( Z_{t_1}^{(1)}>0\bigr) + 
\exp\left\{ - \frac{\theta_1}{2\mu^{(2)}(t_1)}\bigl( Z_{t_1}^{(1)}\bigr)^2\right\}
\one\bigl(   Z_{t_1}^{(1)}<0\bigr) 
$$
if both $t_1=a$ and  $\mu^{(1)}(t_1)=0$, and let $W_a=1$
otherwise. Similarly, we let
$$
W_b= 
\one\bigl( Z_{t_k}^{(1)}<0\bigr) + 
\exp\left\{ - \frac{\theta_1}{2\mu^{(2)}(t_k)}\bigl( Z_{t_k}^{(1)}\bigr)^2\right\}
\one\bigl(   Z_{t_k}^{(1)}>0\bigr) 
$$
if both $t_k=b$ and  $\mu^{(1)}(t_k)=0$, and let $W_b=1$
otherwise. Finally, we let
\begin{equation} \label{e:W}
W= W_{(a,b)}W_aW_b\prod_{j=k+1}^l \one\bigl( Z_{t_j}>0\bigr)\,.
\end{equation}

For the rest of this section we will use the notation $\Sigma$ for the
covariance matrix of the Gaussian vector  $\bigl(X_{t_1}, \ldots,
X_{t_k}\bigr)$ and $V_*$ for the optimal value in the optimization
problem \eqref{p1.eq1}. 

It follows from \ams\ that 
\[
\log \P\left(\min_{t\in[a,b]}X_t>u\right)\sim\log \P\left(\min_{t\in
    S}X_t>u\right), \ u\to\infty\,,
\]
in the sense of \eqref{eq.defsim}. The following
theorem both explains how the two tail probabilities are related once the
logarithms are removed, and answers Question 1.  
\begin{theorem}\label{t.q1}
Suppose that a Gaussian process $\BX$ satisfies either (a) or (b) of Proposition \ref{p1}. Then, as $u\to\infty$,
\begin{equation}
\label{t.q1.eq1}\P\left(\min_{t\in[a,b]}X_t>u\right)\sim\E(W) \P\left(\min_{t\in S}X_t>u\right)\,,
\end{equation}
and hence
\begin{equation}
\label{t.q1.eq2}\P\left( \min_{t\in [a,b]} X_t>u\right) \sim
\frac{\E(W)}{(2\pi)^{k/2}(\theta_1\ldots\theta_k)(\det\Sigma)^{1/2}}u^{-k}
\exp\left(-\frac{u^2}{2V_*}\right)\,, 
\end{equation}
the above equivalences to be interpreted as in \eqref{eq.defsim}, including when $\E(W)=0$.
Furthermore, $\E(W)>0$ if and only if 
\begin{equation}
\label{t2.eq2}
\mu^{(2)}(t)>0\text{ for all }t\in S\cap(a,b)\,.
\end{equation}
\end{theorem}
 
We now proceed to address the rest of the questions in the
introduction. From now on we assume \eqref{t2.eq2} to hold. The
following result answers Question 2. 

\begin{theorem}
\label{t3}
Under assumptions of Theorem \ref{t.q1}, assume also that
\eqref{t2.eq2} holds. Then, as $u\to\infty$,
\[
\P\left((X_t-u\mu(t):a\le t\le b)\in\cdot\left|\min_{t\in
      [a,b]}X_t>u\right.\right)\Rightarrow Q_W(\cdot) 
\]
weakly on $C[a,b]$, where $Q_W$ is a probability measure on $C[a,b]$
defined by
$$
Q_W(B) =\frac{1}{\E W} \E\left[ \one\left( \bigl( Z_t:a\le t\le
b\bigr)\in B\right)W\right], \ B\subset C[a,b], \ \text{Borel,}
$$
with $W$ given by \eqref{e:W}.
\end{theorem}

The next result is an answer to Question 3.

\begin{theorem}
\label{t4}
Under assumptions of Theorem \ref{t.q1}, assume also that
\eqref{t2.eq2} holds. Then, as $u\to\infty$, the conditional
distribution of $u\bigl( \min_{t\in [a,b]}X_t-u\bigr)$  given
$\min_{t\in [a,b]}X_t>u$ converges weakly to an exponential
distribution with mean $V_*$.
\end{theorem}

Finally, we answer Question 4.

\begin{theorem}
\label{t5}
Under assumptions of Theorem \ref{t.q1}, assume also that
\eqref{t2.eq2} holds. Let
\[
T_{*}:=\arg\min_{s\in[a,b]}X_s\,,
\]
where we choose the leftmost location of the minimum in case there are
ties. Then, as $u\to\infty$,
\[
\P\left(T_{*}\in\cdot\, \Bigr|\min_{s\in[a,b]}X_s>u\right)\Rightarrow \nu_*\,,
\]
where $\nu_*$ is the unique minimizer in \eqref{p1.eq1}. 
\end{theorem}

\section{Examples}\label{sec:examples} 

In order to illustrate the general results in Section
\ref{sec:results} we will, in 
this section, look at specific examples of Gaussian
processes satisfying the assumptions of the general results. We start
with a quintessential example of a stationary process.

\begin{ex} \label{ex:Gauss} {\bf Gaussian covariance function} 
{\rm 

Consider the zero mean 
stationary Gaussian process $(X_t:t\in\bbr)$ with covariance function 
\[
R(t):=\E(X_sX_{s+t})=\exp\left(- t^2/2\right), \, t\geq 0\,.
\]
This process has a spectral density that coincides with the standard
Gaussian density, as in \eqref{e:Gauss.sp.dens}, hence the process has properties S1 and S2 of
Section \ref{sec:smoothness}. Therefore, the results in Section
\ref{sec:results}  apply. Recall that some of the results require the
assumption \eqref{t2.eq2}. We will presently see both that this assumption
may fail and that this assumptions fails only rarely. 

Let
\begin{equation}\label{eg.eq1}
0=a<b<c_1:=\min\left\{y>0:2e^{-y^2/8}-1-e^{-y^2/2}=0\right\}\approx2.2079\,.
\end{equation}
It has been shown in Proposition 5.3 and Example 6.1 of \ams\ that for such an interval,
\begin{eqnarray}
\nonumber E=S&=&\{a,b\}\,,\\
\label{eg.eq2}\mu(t)&=&\frac{e^{-t^2/2}+e^{-(b-t)^2/2}}{1+e^{-b^2/2}},
                        \, t\in [a,b]\,, \\
\nonumber V_*&=& \frac{1+e^{-b^2/2}}{2}\,.
\end{eqnarray}
Assumption \eqref{t2.eq2} holds automatically, and 
Theorem \ref{t.q1} implies that 
\[
\P\left(\min_{a\le t\le
    b}X_t>u\right)\sim C_1u^{-2}\exp\left(-\frac1{1+e^{-b^2/2}}u^2\right) 
\]
for $C_1>0$ as $u\to\infty$. In fact, one can check that 
\begin{eqnarray}\label{eg.eq3}
C_1:=\frac1{2\pi}\frac{(1-e^{-b^2})^{3/2}}{(1-e^{-b^2/2})^2}\,.
\end{eqnarray}
By Theorem \ref{t4}, conditionally on the event $\bigl\{ \min_{t\in
  [a,b]}X_{t}>u\bigr\} $, the scaled overshoot $u\bigl( \min_{t\in
  [a,b]}X_{t}-u\bigr)$  converges weakly, as $u\to\infty$, 
 to an exponential random variable with the mean $(1+e^{-b^2/2})/2$.

Next we consider the situation when $a=0$ and $b=c_1$, defined by 
in \eqref{eg.eq1}. Then $\mu$ and $V_*$ are still as in
\eqref{eg.eq2}, but we now have 
$S=\{a,b\}$ and $E=\{a,b/2,b\}$. Theorem \ref{t.q1} still applies, and
it gives 
\begin{equation}
\P\left(\min_{a\le t\le
    b}X_t>u\right)\sim\frac12C_1u^{-2}\exp\left(-\frac1{1+e^{-b^2/2}}u^2\right)\,, 
\end{equation}
as $u\to\infty$, where $C_1$ is as defined in \eqref{eg.eq3}. The
asymptotic conditional distribution of  the scaled overshoot $u\bigl( \min_{t\in
  [a,b]}X_{t}-u\bigr)$  is same as in the case $b<c_1$. 

We proceed to the case  $a=0$ and 
\begin{equation}
\label{eg.eq4}c_1<b<c_2:=\min\left\{y>c_1:(1-\vep(y))\left(\frac{y^2}4-1\right)e^{-y^2/8}
=\vep(y)\right\}\,, 
\end{equation}
with 
\[
\vep(y):=\frac{1+e^{-y^2/2}-2e^{-y^2/8}}{3+e^{-y^2/2}-4e^{-y^2/8}}\,.
\]
The value of $c_2\approx 3.9283$. Proposition 5.5 of \ams\ shows that
in this case 
\begin{align}
E=S=&\{a,b/2,b\}\,,\notag \\
\mu(t)=&\frac{1}{V_*}\left[\frac{1-\vep(b)}2\left(e^{-t^2/2}+e^{-(b-t)^2/2}\right)
+\vep(b)e^{-(t-b/2)^2/2}\right], \, t\in\bbr\,, \notag \\
V_*=&
       \Var\left(\frac{1-\vep(b)}2\left(X_a+X_b\right)+\vep(b)X_{b/2}\right)\,.\label{eg.eq6} 
\end{align}
In this case 
\begin{equation}
\label{eg.eq5}\mu^{(2)}(b/2)>0\,, 
\end{equation}
so that \eqref{t2.eq2} holds. Theorem \ref{t.q1} implies that for some 
$C_2>0$,  
\[
\P\left(\min_{a\le t\le b}X_t>u\right)\sim C_2u^{-3}\exp\left(-\frac1{2V_*}u^2\right)\,,
\]
as $u\to\infty$. 
The $b$-dependent constant $C_2$  can be explicitly calculated if
desired. The
asymptotic conditional distribution of  the scaled overshoot $u\bigl( \min_{t\in
  [a,b]}X_{t}-u\bigr)$  is exponential   with mean $V_*$. 

In the case when $a=0$ and $b=c_2$, which are as in \eqref{eg.eq4}, then
\eqref{eg.eq6} still holds, but \eqref{eg.eq5} fails, and therefore the
assumption \eqref{t2.eq2} no longer holds.  
In this situation,   Theorem \ref{t.q1} only says that 
\[
\P\left(\min_{a\le t\le
    b}X_t>u\right)=o\left(u^{-3}\exp\left(-\frac1{2V_*}u^2\right)\right)\,,
\]
as $u\to\infty$. 

\begin{figure}[!ht]
\includegraphics[scale=0.6]{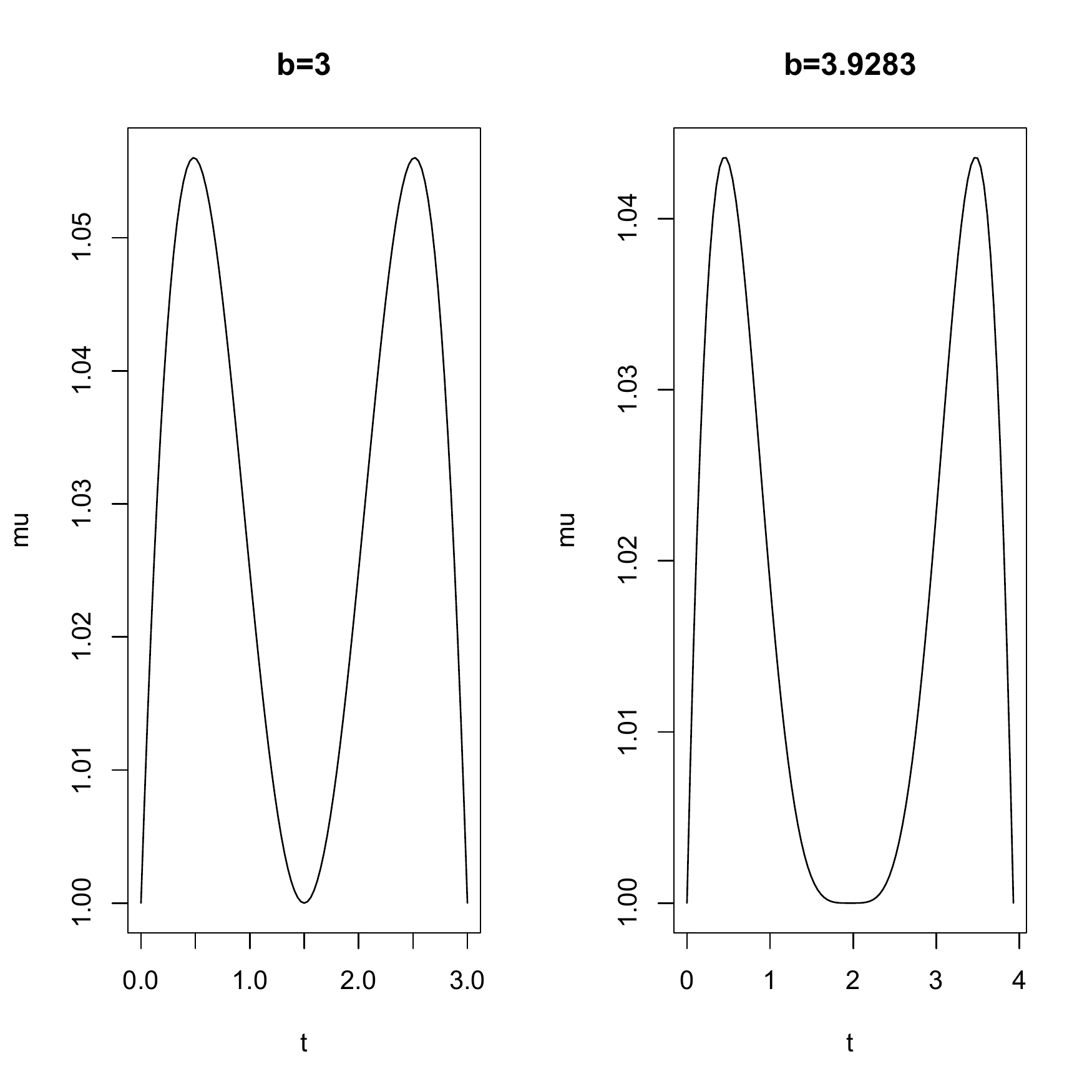}
\caption{Function $\mu$ for $b=3$ and $b=3.9283$ in Example \ref{ex:Gauss}}
\label{fig:case.1}
\end{figure}

The plots on Figure \ref{fig:case.1} illustrate the
different behaviour of the function $\mu$ when $c_1<b<c_2$, and when $b=c_2$.

The phenomenon exhibited in Example \ref{ex:Gauss} is very general and
holds for all smooth stationary Gaussian we have checked. For example,
for the Gaussian process with the uniform spectral measure and
covariance function $R(t)=\sin t/t, \, t\geq 0$ one gets a nearly
identical behaviour to the one observed for the process with the
Gaussian covariance function. The breakpoints are $c_1\approx 4.275$
and $c_2\approx 9.365$. 
}
\end{ex}

If a stationary Gaussian process satisfies S1 and S2, then the unique
optimizer $\nu_*$ of the minimization problem \eqref{p1.eq1} must be
symmetric around the midpoint of the interval $[a,b]$; indeed, for
any probability measure $\nu$ supported by $[a,b]$, the measure
$\tilde \nu$ obtained by reflection around the midpoint leads to the
same value in the integral and hence, by convexity,
$\nu_*=(\nu_*+\tilde\nu_*)/2$ is symmetric. Since the optimal measure
cannot be concentrated at the midpoint of the interval, we conclude
that the cardinality of the set $S$ in this case is at least 2.
However, for a non-stationary process, $S$ may be a singleton. The
next example illustrates this fact.

\begin{ex} \label{ex:sing} {\bf A non-stationary process} 
{\rm  

We start with a stationary centered Gaussian process $\BZ=(Z_t:t\in\bbr)$ with a spectral
measure $F_X$ satisfying S1 and S2.  
Let $Y$ be a standard normal random variable independent of $\BZ$. Define
\[
X_t:=Y(1+t^2-t^4)+Z_t-Z_0,\, t\in\bbr\,.
\]
Let $-1<a<0<b<1$. Note that the process $\BX:=(X_t:t\in\bbr)$ has
Property \ref{assume1}  by construction. It is elementary that the
covariance function $R$ of the process $\BX$ satisfies 
\eqref{assume4}. 

Clearly, for any probability measure $\nu$ on $[a,b]$,
\begin{equation} \label{e:lower.b.int}
\int_a^b\int_a^bR(s,t)\, \nu(ds)\,
\nu(dt)\ge\left(\int_a^b(1+t^2-t^4)\, \nu(dt)\right)^2\ge1\,.
\end{equation} 
Therefore, the process $\BX$ has Property \ref{assume2}. In order to
check that it also has Property \ref{assume3} , suppose that 
for distinct reals numbers $t_0,\ldots,t_k$ and  coefficients
$\alpha_0,\ldots,\alpha_k,\beta_0,\ldots,\beta_k$, we have 
\[
\sum_{j=0}^k\left(\alpha_jX_{t_j}+\beta_jX_{t_j}^{(1)}\right)=0\text{ almost surely}\,.
\]
Without loss of generality, we assume $t_0=0$. Using the independence
of $Y$ and $\BZ$ we see that 
\[
-Z_0\sum_{j=1}^k\alpha_j+\sum_{j=1}^k\alpha_jZ_{t_j}+\sum_{j=0}^k\beta_jZ_{t_j}^{(1)}=0\text{ almost surely}\,.
\]
Proposition \ref{p0} implies that
$\alpha_1=\ldots=\alpha_k=\beta_0=\ldots=\beta_k=0$. Thus, also
$\alpha_0=0$, and so process $\BX$ has Property \ref{assume3}.

Note that the choice $\nu=\delta_{\{0\}}$ is, by
\eqref{e:lower.b.int}, the optimal measure $\nu_*$. Thus,
$S=E=\{0\}$. That is, the support $S$  is a singleton
which contains none of the endpoints of the interval. That is, the
conclusion of Proposition \ref{p2} indeed fails without appropriate
assumptions on the covariance function of the process. 

Here
\[
\mu(t)= \E(X_t|X_0=1) = \E(X_t|Y=1) = 
1+t^2-t^4,\, t\in\bbr\,. 
\]
Therefore, $\mu^{(2)}(0)=2>0$. By Theorem \ref{t.q1},  
\[
\P\left(\min_{a\le t\le b}X_t>u\right)\sim\frac1{\sqrt{2\pi(1+\lambda_2/2)}}u^{-1}e^{-u^2/2}\,,
\]
as $u\to\infty$, where $\lambda_2$ is the second spectral moment of
$\BZ$. By Theorem \ref{t4}, conditionally on the event $\bigl\{ \min_{t\in
  [a,b]}X_{t}>u\bigr\} $, the scaled overshoot $u\bigl( \min_{t\in
  [a,b]}X_{t}-u\bigr)$  converges weakly, as $u\to\infty$, to the
standard exponential random variable. 
}
\end{ex}

\section{Proofs}\label{sec:proofs}  

We start with several preliminary results. First, an elementary
convergence statement, whose proof is omitted. 
\begin{lemma}\label{l3}
Suppose that three families of random variables ${\mathbb
  U}:=(U_{iN}:1\le i\le m,1\le N\le\infty)$, ${\mathbb
  V}:=(V_{iN}:1\le i\le n,1\le N<\infty)$ and ${\mathbb W}:=(W_i:1\le
i\le n)$ live on the same probability space, and that $\mathbb U$ and
$\mathbb W$ are independent. Suppose that, as $N\to\infty$, 
\begin{eqnarray*}
(U_{1N},\ldots,U_{mN})&\Rightarrow&(U_{1\infty},\ldots,U_{m\infty})\,,\\
V_{iN}&\prob&W_i\,,1\le i\le n\,.
\end{eqnarray*}
Then 
\[
\left(U_{1N},\ldots,U_{mN},V_{1N},\ldots,V_{nN}\right)\Rightarrow
\left(U_{1\infty},\ldots,U_{m\infty},\,   W_{1},\ldots,W_{n}\right) 
\]
as as $N\to\infty$, where the random vectors
$(U_{1\infty},\ldots,U_{m\infty})$ and $(W_{1},\ldots,W_{n})$ in the
right hand side are independent. 
\end{lemma}

The next lemma is the first step in the proof of Theorem \ref{t.q1}. 
\begin{lemma}\label{l1}
Under the assumptions of Theorem \ref{t.q1}, 
for all $\vep>0$ and $n\ge0$ we have 
\begin{equation}
\label{l1.eq1}\lim_{u\to\infty}\P\left(\left.\sup_{a\le t\le
      b}\left|X^{(n)}_t-u\mu^{(n)}(t)-Z^{(n)}_t\right|>\vep\right| \min_{s\in
  S} X_s>u\right)=0\,.
\end{equation}
\end{lemma}
\begin{proof}
Notice that we can write 
\[
Y_t^{(n)} =\sum_{j=1}^k\alpha_j(t)X_{t_j}\,,
\]
$t\in \bbr$, 
for some continuous functions $\alpha_1,\ldots,\alpha_n$ from $\bbr$ to $\bbr$. Therefore,
\[
\mu^{(n)}(t)=\sum_{j=1}^k\alpha_j(t)\,,t\in\bbr\,,
\]
and probability in the left hand side of \eqref{l1.eq1} equals
\begin{eqnarray*}
&&\P\left(\left.\sup_{a\le t\le b}\left|Y^{(n)}_t-u\mu^{(n)}(t)\right|>\vep\right|\min_{s\in
  S} X_s>u\right)\\
&=&\P\left(\left.\sup_{a\le t\le b}\left|\sum_{j=1}^k\alpha_j(t)(X_{t_j}-u)\right|>\vep\right|\min_{s\in
  S} X_s>u  \right)\\
&\le&\P\left(\left.\sum_{j=1}^k|X_{t_j}-u|>\frac\vep{\max_{a\le t\le
      b,1\le j\le k}|\alpha_j(t)|}\right| \min_{s\in
  S} X_s>u \right)\\
&\to&0\,,
\end{eqnarray*}
as $u\to\infty$ because, conditionally on $\{\min_{s\in S} X_s>u\}$,
$X_{t_j}-u\prob0$ for each $j=1,\ldots, k$ by part (ii) of Proposition \ref{l:Gauss.vector}. 
\end{proof}

The next theorem  is the crucial step towards proving Theorem
\ref{t.q1}. Its statement uses Lemma \ref{l:mu.prop}. Below and elsewhere, we follow the convention 
$\infty\cdot0=0\cdot\infty=0$.

\begin{theorem} \label{t1}
Suppose that the assumptions of Theorem \ref{t.q1} are satisfied, and
let $t\in S$. 

(i)  Suppose that $t\in (a,b)$. Let 
$\vep>0$ be such that $[t-\vep,t+\vep]\subset[a,b]$ and
\begin{equation}
\label{t1.assume1}\mu^{(2)}(s)>0\text{ for all }0<|s-t|\le\vep\,.
\end{equation}
Then, as $u\to\infty$, conditionally on the event $\{\min_{s\in S} X_s>u\}$, 
\begin{equation} \label{e:ins.S}
u\left(X_t-\min_{s\in[t-\vep,t+\vep]}X_s\right)
\prob\frac1{2\mu^{(2)}(t)}\left(Z_t^{(1)}\right)^2\,,
\end{equation}
and, as before, we interpret 
the right hand side as $+\infty$ if $\mu^{(2)}(t)=0$.

(ii)  Suppose that $t=a$. Then either $\mu^{(1)}(a)>0$  or
$\mu^{(1)}(a)=0$ and $\mu^{(2)}(a)\geq 0$. Then for all $\vep>0$
small enough, as $u\to\infty$, conditionally on the
event $\{ \min_{s\in S} X_s>u\}$, 
\begin{eqnarray*}
&&u\left(X_a-\min_{s\in[a,a+\vep]}X_s\right)\\
&\prob&\begin{cases}
0\,,&\mu^{(1)}(a)>0\,,\\
\frac1{2\mu^{(2)}(a)}\left(Z_a^{(1)}\right)^2\one\left(Z_a^{(1)}<0\right)\,,&\mu^{(1)}(a)=0,
\, \mu^{(2)}(a)>0\,,\\
\infty\one\left(Z_a^{(1)}<0\right)\,,&\mu^{(1)}(a)=0, \, \mu^{(2)}(a)=0\,.
\end{cases}
\end{eqnarray*}

(iii)   Suppose that $t=b$. Then either $\mu^{(1)}(b)<0$  or
$\mu^{(1)}(b)=0$ and $\mu^{(2)}(b)\geq 0$. Then for all $\vep>0$ small
enough, as $u\to\infty$, conditionally on the
event $\{ \min_{s\in S} X_s>u\}$, 
\begin{eqnarray*}
&&u\left(X_b-\min_{s\in[b-\vep,b]}X_s\right)\\
&\prob&
\begin{cases}
0\,,&\mu^{(1)}(b)<0\,,\\
\frac1{2\mu^{(2)}(b)}\left(Z_b^{(1)}\right)^2\one\left(Z_b^{(1)}>0\right)\,,&\mu^{(1)}(b)=0,
\, \mu^{(2)}(b)>0\,,\\
\infty\one\left(Z_b^{(1)}>0\right)\,,&\mu^{(1)}(b)=0, \, \mu^{(2)}(b)=0\,.
\end{cases}
\end{eqnarray*}
\end{theorem}
\begin{proof}
The claims of parts (ii) and (iii) are similar, so we will only prove
the claims of parts (i) and (ii). We start with part (i). Fix $t\in S\cap(a,b)$ and $\vep>0$ such that $[t-\vep,t+\vep]\subset[a,b]$ and \eqref{t1.assume1} holds. Define
\[
t_*:=\arg\min_{t-\vep\le s\le t+\vep}X_s\,,
\]
taken to be the closest to $t$ location of the minimum in case there are
ties. We first check that
\begin{equation}
\label{t1.eq1}\lim_{u\to\infty}\P\left(\left.X^{(1)}_{t_*}=0\right|\min_{s\in
    S}S_s>u\right)=1\,.
\end{equation}
To see this, note that by Lemma \ref{l1}, 
\begin{eqnarray*}
\P\left(\left.X^{(1)}_{t_*}=0\right|\min_{s\in S}X_s>u\right)&\ge&
\P\left(\left.X^{(1)}_{t+\vep}>0,\,X^{(1)}_{t-\vep}<0\right|\min_{s\in S}X_s>u\right)\\
&\to&1\,,
\end{eqnarray*}
as $u\to\infty$. Here we have used the fact that by
\eqref{t1.assume1}, 
\[
\mu^{(1)}(t-\vep)<0<\mu^{(1)}(t+\vep)\,.
\]
Keeping the definition of $t_*$ unchanged, but replacing $\vep$ by
arbitrarily small $0<\vep^\prime<\vep$ in the above argument, shows
that, as $u\to\infty$, conditionally on the event $\{ \min_{s\in S}X_s>u\}$,
\begin{equation}
\label{t1.eq2} t_*\prob t\,.
\end{equation}
By Lemma \ref{l:mu.prop}  there exists an even positive integer $n$ such that
\begin{equation}
\label{t1.eq3}\mu^{(1)}(t)=\ldots=\mu^{(n-1)}(t)=0<\mu^{(n)}(t)\,.
\end{equation}
Consider the series expansion 
\begin{equation}
\label{t1.eq5}X^{(1)}_{t_*}=X_t^{(1)}+\sum_{j=2}^{n-1}X^{(j)}_t\frac{(t_*-t)^{j-1}}{(j-1)!}+X^{(n)}_{\xi_1}\frac{(t_*-t)^{n-1}}{(n-1)!}\,,
\end{equation}
for some $\xi_1$ in between $t$ and $t_*$. By Lemma \ref{l1} and
\eqref{t1.eq3} we know that, as $u\to\infty$, conditionally on the
event $\{ \min_{s\in S}X_s>u\}$, 
\begin{equation}
\label{l1.cons1}X_t^{(j)}\prob Z_t^{(j)}\,,1\le j\le n-1\,.
\end{equation}

The above along with \eqref{t1.eq1}, \eqref{t1.eq2} and \eqref{t1.eq5}
imply that, as $u\to\infty$, conditionally on the
event $\{ \min_{s\in S}X_s>u\}$, 
\[
X^{(n)}_{\xi_1}\frac{(t_*-t)^{n-1}}{(n-1)!}\prob-Z_t^{(1)}\,.
\]
Furthermore, by Lemma \ref{l1}, we also have 
\begin{equation}
\label{l1.cons2}u^{-1}X^{(n)}_{\xi_1}\prob\mu^{(n)}(t)\,.
\end{equation}
Therefore, 
\begin{equation}
\label{t1.eq6}u(t_*-t)^{n-1}\prob-Z_t^{(1)}\frac {(n-1)!}{\mu^{(n)}(t)}
\,.
\end{equation}
Next, we use the series expansion 
\begin{equation} \label{e:sec.exp}
X_{t_*}-X_t=\sum_{j=1}^{n-1}X^{(j)}_t\frac{(t_*-t)^{j}}{j!}+X^{(n)}_{\xi_2}\frac{(t_*-t)^{n}}{n!}\,,
\end{equation} 
for some $\xi_2$ in between $t$ and $t_*$. Since \eqref{l1.cons2} also
holds with $\xi_2$ replacing $\xi_1$, we conclude by \eqref{l1.cons1}
with $j=1$ and \eqref{t1.eq6} that 
\begin{eqnarray}
&&u^{1/(n-1)}\left(X_{t_*}-X_t\right) \notag \\
&=&u^{1/(n-1)}(t_*-t)X_t^{(1)}+\left(u^{-1}X^{(n)}_{\xi_2}\right)\frac{u^{n/(n-1)}(t_*-t)^n}{n!}\notag
  \\
&&+u^{1/(n-1)}\sum_{j=2}^{n-1}X^{(j)}_t\frac{(t_*-t)^{j}}{j!}\notag \\
&\prob&-|Z^{(1)}_t|^{n/(n-1)}\frac1{(\mu^{(n)}(t))^{1/(n-1)}}((n-1)!)^{1/(n-1)}\frac{n-1}n
        \,. \label{e:two.exp}
\end{eqnarray}
When $n=2$, this reduces to \eqref{e:ins.S}.
If $n>2$, i.e. if $\mu^{(2)}(t)=0$, the above limit says that 
\begin{equation*}
u(X_{t_*}-X_t)\prob-\infty\,,
\end{equation*}
which is, again, \eqref{e:ins.S}. This completes the proof of part
(i). 

We now prove part (ii) of the theorem. 
The claim will be proved separately for the three cases listed in the
statement. We start with the case   $\mu^{(1)}(a)>0$. For all 
$\vep>0$ small enough, such that $a+\vep\leq b$ and 
$\min_{s\in[a,a+\vep]}\mu^{(1)}(s)>0$, we have by Lemma \ref{l1}, 
\begin{eqnarray*}
  &&\P\left(\left.X_a=\min_{s\in[a,a+\vep]}X_s\right|\min_{s\in S}
  X_s>u\right)\\
  &\ge&\P\left(\left.\min_{s\in[a,a+\vep]}X_s^{(1)}>0\right|\min_{s\in
                    S}X_s>u\right)\\ 
&\to&1\,,
\end{eqnarray*}
as $u\to\infty$, which proves the claim of part (ii) in the case
$\mu^{(1)}(a)>0$. 

Suppose now that $\mu^{(1)}(a)=0$.  By Lemma \ref{l:mu.prop}, we can
choose  $\vep>0$ such that
\begin{equation}
\label{t1.eq8}\mu^{(2)}(s)>0\text{ for all }a<s\le a+\vep\,.
\end{equation}
Consider the event
\[
B:=\left\{X^{(1)}_a>0>\min_{s\in[a,a+\vep]}X_s^{(1)}\right\}\,.
\]
Our first claim is that
\begin{equation}
\label{t1.eq9}\lim_{u\to\infty}\P\left( B\left|\min_{s\in S}X_s>u\right.\right)=0\,.
\end{equation}
Indeed, on the event $B$,  the derivative $X^{(1)}_{s}$ crosses 0 in
the interval  $[a,a+\vep]$. If we define a random variable 
\[
s_*:=\inf\{s\in[a,a+\vep]:X^{(1)}_{s}=0\}\,,
\]
then 
\begin{equation}
\label{t1.eq10}X^{(1)}_{s_*}\, \one_B=0\,.
\end{equation}
Furthermore, the definition of $s_*$ tells us that
\begin{equation}
\label{t1.eq11}\left(X_{s_*}-X_a\right)\one_B\ge0\,.
\end{equation}
Note that by  \eqref{t1.eq8} the second derivative $\mu^{(2)}$ is
bounded away from 0 on any interval $[a+\delta, a+\vep]$ for
$0<\delta<\vep$. It follows from Lemma \ref{l1} that 
\[
\P\left( \left.\min_{a+\delta\leq t\leq a+\vep}X^{(2)}_t<0\right|
  \min_{s\in S}X_s>u\right) \to 0
\]
 as $u\to\infty$. Therefore, conditionally on the event $\{ \min_{s\in
   S}X_s>u\}$, as $u\to\infty$, 
\[
(s_*-a)\, \one_B\prob0\,.
\]
Using twice the Taylor expansion and imitating the steps leading to
\eqref{t1.eq6} and \eqref{e:two.exp}, in conjunction with
\eqref{t1.eq10}, shows that 
\[
\lim_{u\to\infty}\P\left(X_{s_*}<X_a\Bigl| B\cap\{\min_{s\in  S}\{X_s>u\}
\right)=1\,.
\]
This would contradict \eqref{t1.eq11} if \eqref{t1.eq9} were false. Thus \eqref{t1.eq9} follows. 

Write
\begin{eqnarray*}
&&u\left(X_a-\min_{s\in[a,a+\vep]}X_s\right)\\
&=&u\left(X_a-\min_{s\in[a,a+\vep]}X_s\right)\one(X_a^{(1)}<0)\\
&+&u\left(X_a-\min_{s\in[a,a+\vep]}X_s\right)\one(\{X_a^{(1)}>0\}\setminus B)\\
&+&u\left(X_a-\min_{s\in[a,a+\vep]}X_s\right)\one(\{X_a^{(1)}>0\}\cap B)\,.
\end{eqnarray*}
By the definition of the event $B$, the middle term in the right
hand side is equal to zero, while by 
\eqref{t1.eq9}, the last term in the right hand side goes to zero
in probability.  It remains, therefore, to consider the first term in the right
hand side. By the assumption $\mu^{(1)}(a)=0$ and Lemma  \ref{l1} we
know that $\one(X_a^{(1)}<0)\to \one(Z_a^{(1)}<0)$ in probability. For
the rest of the that term the same analysis as the one used in the
proof of part (i) applies. Specifically, we use the two Taylor
expansions \eqref{t1.eq5} and \eqref{e:sec.exp}. The only difference
between the two scenarios is that now the integer $n$ does not need to
be an even number, but it plays no role in the argument. 

This completes the proof of the theorem in all cases. 
\end{proof}



We have now all the ingredients needed to prove Theorem \ref{t.q1}. 

\begin{proof}[Proof of Theorem \ref{t.q1}]
A restatement of \eqref{t.q1.eq1} is
\begin{equation}\label{t2.eq1}
\lim_{u\to\infty}\P\left(\min_{t\in [a,b]}X_t>u\Big|\min_{t\in S}X_t>u\right)=\E W\,,
\end{equation}
which we proceed to show first.
For $u>0$ let $(\tilde Y_t^{(u)}:t\in\bbr)$ be a process with
continuous sample paths whose law is the law of the process $\BY$ in
\eqref{eq.defy} conditioned on the event $\{ \min_{s\in S}X_s>u\}$, and
let this process be independent of the process $\BZ$ in \eqref{eq.defz}.
 Define
\[
\tilde X_t^{(u)}:=\tilde Y_t^{(u)}+Z_t, \, t\in\bbr\,.
\]
Let $\vep>0$ be small enough such that the convergence in probability
in Theorem \ref{t1} holds. Continuing using the notation $S=\{
t_1,\ldots, t_k\}$ and $E\setminus S=\{ t_{k+1}, \ldots, t_{k+l}\}$, 
we define for $1\le j\le k$ 
\begin{equation*}
V_{ju}:=u\left(\tilde
           X^{(u)}_{t_j}-\min_{s\in[t_j-\vep,t_j+\vep]\cap[a,b]}\tilde
           X_s^{(u)}\right)\,, 
\end{equation*}
and 
\[
V_{k+1,\,u}:=
\inf_{s\in G}\left[\tilde X^{(u)}_{s}-u\mu(s)\right]\,,
\]
where
\begin{equation}
\label{eq.defg}G:=\{s\in[a,b]:|s-t_j|\le\vep\text{ for some }k+1\le j\le l\}\,,
\end{equation}
with the convention that infimum over the empty set is defined as
$-\infty$.  

For $j=1,\ldots, k$ we denote by $W_j$ (not to be confused with $W_a$, $W_b$ or $W_{(a,b)}$) the limit in probability of
\[
T_{ju}:=u\left(X_{t_j}-\min_{s\in[t_j-\vep,t_j+\vep]\cap[a,b]} X_s\right)\,,
\]
as $u\to\infty$, conditionally on the event $\{\min_{s\in S}X_s>u\}$,
given in Theorem \ref{t1}.  Recall that $W_j$ may take the value
$+\infty$. We define also 
\[
U_{ju}:=u\bigl(\tilde X^{(u)}_{t_j}-u\bigr)\,,1\le j\le k\,.
\]

Clearly, the conditional law of 
\[
 \Bigl(u(X_{t_1}-u),\ldots,u(X_{t_k}-u),T_{1u},\ldots,T_{ku},\, \min_{s\in
   G}\left[ X_{s}-u\mu(s)\right]\Bigr) 
\]
given $\{\min_{s\in S}X_s>u\}$ coincides with the law of 
$$
\Bigl(U_{1u},\ldots,U_{ku},V_{1u},\ldots,V_{k+1,\,u}\Bigr)\,.
$$
By Theorem \ref{t1} we know that for fixed $1\le j\le k$, 
\begin{equation} \label{t2.eq4}
  V_{ju}\prob W_j 
\end{equation}
as $u\to\infty$, where we regard $W_j$ as a function of the process
$\BZ$ in the definition of $(V_{ju})$. Furthermore, by 
Proposition \ref{l:Gauss.vector}, 
\begin{equation}
(U_{1u},\ldots,U_{ku})\Rightarrow(E_1,\ldots,E_k)
\end{equation}
as $u\to\infty$, where $E_1,\ldots, E_k$ are independent exponential
random variables with parameters given by \eqref{e:deftheta}. 
Finally, Lemma \ref{l1} implies that as $u\to\infty$,
\begin{equation}
\label{t2.eq5}V_{k+1,\,u}\prob\min_{s\in G}Z_s\,.
\end{equation}
We apply now Lemma \ref{l3} to conclude that, as $u\to\infty$,
\[
\Bigl(U_{1u},\ldots,U_{ku},V_{1u},\ldots,V_{k+1,\,u}\Bigr)\Rightarrow\Bigl(E_1,\ldots,E_k,W_1,\ldots,W_k,\min_{s\in G}Z_s\Bigr)\,,
\]
with $(E_1,\ldots, E_k)$ being independent of the rest of the random
variables in the right hand side, which implies that, as $u\to\infty$,   
\[
\Bigl(u(X_{t_1}-u)-T_{1u},\ldots,u(X_{t_k}-u)-T_{ku},\, \min_{s\in
  G}\left[ X_{s}-u\mu(s)\right]\Bigr|\min_{s\in S}X_s>u\Bigr)
\]
\begin{equation}
\label{t2.eq6}
\Rightarrow\Bigl(E_1-W_1,\ldots,E_k-W_k,\, \min_{s\in G}Z_s\Bigr)\,,
\end{equation}
weakly on $\bbr^{k+1}$, with the obvious interpretation if some of the
$W_j$ take the value $+\infty$. If we denote 
\begin{equation}
\label{eq.defh}H:=\{s\in[a,b]:|s-t_j|\le\vep\text{ for some }1\le j\le
k\}\,, 
\end{equation}
then it follows by the continuous mapping theorem that, as $u\to\infty$,
\begin{eqnarray}
\nonumber&&\left(\left.\min\left\{u\left(\min_{s\in H}X_s-u\right),\, \min_{s\in G}\left[
    X_{s}-u\mu(s)\right]\right\} \right|\min_{s\in S}X_s>u\right)\\
\label{eq.conv}&\Rightarrow&\min\left\{E_1-W_1,\ldots,E_k-W_k,\,
  \min_{s\in G}Z_s\right\}\,. 
\end{eqnarray}
Note that \eqref{eq.conv} continues to hold if we use $\vep=0$ in the
definition of $G$ (but not in the definition of $H$). 

Since the function $\mu$ is bounded away from 1 on  $[a,b]\setminus
(G\cup H)$, Lemma \ref{l1} implies that 
\[
\lim_{u\to\infty}\P\left(\min_{s\in[a,b]\setminus(G\cup
    H)}X_s>u\Bigr|\min_{s\in S}X_s>u\right)=1\,.
\]
Together with the fact that $\mu(s)\ge1$ for all $s$, this implies that 
\begin{align*}
&\liminf_{u\to\infty}\P\left(\min_{t\in [a,b]}X_t>u\Bigr|\min_{s\in S}X_s>u\right)\\
=&\liminf_{u\to\infty}\P\left(\min_{s\in G\cup H}X_s>u\Bigr|\min_{s\in
    S}X_s>u\right)\\
\ge&\lim_{u\to\infty}\P\left(\min\left\{u\left(\min_{s\in
      H}X_s-u\right),\, \min_{s\in G}\left[
      X_{s}-u\mu(s)\right]\right\}>0\Bigr|\min_{s\in S} X_s>u\right)\\
\geq &\P\left(E_1-W_1>0,\ldots, E_k-W_k>0, \, \min_{s\in
        G}Z_s>0\right) \\
=&  \E\left[\exp\left(-\sum_{j=1}^k\theta_jW_j\right)\one\left( \min_{s\in
        G}Z_s>0\right)\right]      \\
=&  \E\left[ W_{(a,b)}W_aW_b\, \one\left( \min_{s\in
        G}Z_s>0\right) \right]\,.
\end{align*}
 
We let now $\vep\downarrow0$ and use the monotone convergence theorem to
conclude that 
\[
\liminf_{u\to\infty}\P\left(\min_{t\in [a,b]}X_t>u\Bigr|\min_{t\in S}X_t>u\right) \ge \E W\,.
\]
On the other hand, 
\begin{eqnarray*}
&& \P\left(\min_{t\in [a,b]}X_t>u\Bigr|\min_{t\in S}X_t>u\right) \\
&\le&\P\left(\min_{s\in H\cup(E\setminus S)}X_s>u\Bigr|\min_{s\in S}X_s>u\right)\\
&=&\P\left(\min\left\{u\left(\min_{s\in H}X_s-u\right),\min_{s\in
    E\setminus S}\left[ X_{s}-u\mu(s)\right]\right\}>0\Bigr|\min_{s\in
    S}X_s>u\right)\\
&\to&\E W\,,
\end{eqnarray*}
the limit in the last line following from  \eqref{eq.conv} with 
   $\vep=0$ in the definition of $G$ and the fact that by Property 3
   the Gaussian random variables $Z_s,\, s\in E\setminus S$ are
   nondegenerate. Thus, \eqref{t2.eq1} follows. 


In view of \eqref{t.q1.eq1} and part (iii) of Proposition
\ref{l:Gauss.vector}, all that needs to be shown for \eqref{t.q1.eq2} is that 
\begin{equation}
\label{t.q1.q2}\sum_{i=1}^k\theta_i=\frac1{V_*}\,,
\end{equation}
where $V_*$ is the optimal value in \eqref{p1.eq1} which is strictly
positive by Property 2.  However, Theorem 5.1 of \ams\ implies that
\[
\min_{y\in\bbr^k:\min_{1\le i\le k}y_i\ge1}y^T\Sigma^{-1}y=\frac1{V_*}\,,
\]
and the unique minimizer is $\one$. This in conjunction with
\eqref{e:deftheta} establishes \eqref{t.q1.q2}. This completes the
proof. 

For the final claim,  recall from the definition that $W_{(a,b)}=0$ a.s.\ if  \eqref{t2.eq2} fails. It immediately follows that $\E W>0$ implies \eqref{t2.eq2}. For the converse, that is, the `if' part, suppose that \eqref{t2.eq2} holds. Then, $W_{(a,b)}>0$ a.s.\ . Property 3 implies that the collection $(Z_t:t\in\bbr\setminus S)\cup(Z_t^{(1)}:t\in\bbr)$ is linearly independent.
The random vector $(Z_a^{(1)},Z_b^{(1)},Z_{t_{k+1}},\ldots,Z_{t_l})$
has a multivariate normal law.  The linear independence implies that
\[
\P\left(Z_a^{(1)}>0,Z_b^{(1)}<0,\min_{k+1\le j\le l}Z_{t_j}>0\right)>0\,.
\]
It is trivial to check from \eqref{e:W} that on this event, $W=W_{(a,b)}$. Thus, the `if' part follows, which completes the proof.
\end{proof}

\begin{proof}[Proof of Theorem \ref{t3}]
Fix a Borel subset $B$ of $C[a,b]$ such that
\[
\P\left((Z_t:a\le t\le b)\in \partial B\right)=0\,,
\]
where $\partial B$ denotes the boundary of $B$ in the supremum norm
topology, and write 

\begin{align*}
 &\P\left((X_t-u\mu(t):a\le t\le b)\in B\left|\min_{t\in
                [a,b]}X_t>u\right.\right)\\ 
\nonumber=&\frac{\P\left((X_t-u\mu(t):a\le t\le b)\in B,\, \min_{t\in
            [a,b]}X_t>u\left|\min_{t\in 
            S}X_t>u\right.\right)}{\P\left( \min_{t\in [a,b]}X_t>u\left|\min_{t\in
            S}X_t>u\right.\right)}\,.
\end{align*}
The denominator converges to $\E W$ 
by Theorem \ref{t.q1}, as $u\to\infty$, and it is positive since 
\eqref{t2.eq2} is assumed. Furthermore, the same argument as the one
used in the proof of  Theorem \ref{t.q1} gives us 
\begin{eqnarray*}
&&\lim_{u\to\infty} \P\left((X_t-u\mu(t):a\le t\le b)\in B,\, \min_{t\in
            [a,b]}X_t>u\left|\min_{t\in 
            S}X_t>u\right.\right)  \\
&=&  \E\left[ \one\left( \bigl( Z_t:a\le t\le
b\bigr)\in B\right)W\right]\,,
\end{eqnarray*}
and the statement of the theorem follows. 
\end{proof}

Theorems \ref{t4} and \ref{t5} are both based on the following result
that we prove first. 

\begin{theorem}\label{t4.new}
Under assumptions of Theorem \ref{t.q1}, assume also that
\eqref{t2.eq2} holds. For $\vep>0$ define
\[
M_{j\vep}:=\min_{s\in[t_j-\vep,t_j+\vep]\cap[a,b]}X_s, \, 1\le j\le k\,.
\]
Then
\begin{equation}
\label{t4.new.claim1}\lim_{\vep\downarrow0}
\liminf_{u\to\infty}\P\left(\left.\min_{1\le j\le k}M_{j\vep}=
    \min_{t\in [a,b]} X_t\right|  \min_{t\in [a,b]} X_t>u\right)=1\,.
\end{equation}
Furthermore, for $\vep>0$ small enough so that the convergence in all
parts of Theorem \ref{t1} holds,
as $u\to\infty$, conditionally on the event $\bigl\{ \min_{t\in [a,b]}
X_t>u\bigr\}$,  
\begin{equation}
\label{t4.new.claim2}\left(u(M_{1\vep}-u),\ldots,u(M_{k\vep}-u)\right)
\Rightarrow(E_1,\ldots,E_k)\,, 
\end{equation}
where $E_1,\ldots, E_k$ are independent exponential random variables
with respective parameters $\theta_1,\ldots, \theta_k$. 
\end{theorem}

\begin{proof}
With the notation  $E\setminus S=\{ t_{k+1}, \ldots, t_{k+l}\}$ as
above and  the set $G$ defined in \eqref{eq.defg}, we first prove
that   
\begin{equation}
\label{t4.new.eq1} \lim_{\vep\downarrow0}\liminf_{u\to\infty}\P\left(\left.\min_{1\le
      j\le k}M_{j\vep}<\min_{s\in G}X_s\right| \min_{t\in [a,b]}X_t>u\right)=1
\end{equation}
 (note that the definition of $G$  depends on $\vep>0$).   

Let $\vep>0$ be small enough so that the convergence in all
parts of Theorem \ref{t1} holds.  As in the proof of Theorem \ref{t.q1},
we denote by  $W_j$ the limit in probability of 
\[
u\left(X_{t_j}-\min_{s\in[t_j-\vep,t_j+\vep]\cap[a,b]}X_s\right)\,,
\]
 as $u\to\infty$, conditionally on the event $\bigl\{ \min_{t\in
     S}X_t>u\bigr\}$.   It follows from \eqref{t2.eq6} that
\begin{align} \notag
& \ \P\left[\left.\left({u\left(\min_{1\le j\le
            k}M_{j\vep}-u\right)},\, {\min_{s\in
          G}[X_s-u\mu(s)]}\right)\in\cdot\right| \min_{t\in S}X_t>u\right] 
\\
\notag\Rightarrow & \ \P\biggl[ \left({\min_{1\le j\le
                                k}(E_j-W_j)},\, {\min_{s\in
                                 G}Z_s}\right)\in\cdot \biggr] \notag
\end{align}          
weakly in $\bbr^2$, as $u\to\infty$, which is almost a restatement of \eqref{eq.conv}.   Proceeding as in the proof of Theorem \ref{t3}, the minimum over $S$ in the conditioning event can be replaced by that over $[a,b]$, at the cost of appropriate corrections on the right hand side. Therefore,
it can be argued that
 \begin{align} \notag
& \ \P\left[\left.\left({u\left(\min_{1\le j\le
            k}M_{j\vep}-u\right)},\, {\min_{s\in
          G}[X_s-u\mu(s)]}\right)\in\cdot\right| \min_{t\in [a,b]}X_t>u\right] 
\\
 \label{t4.new.eq2}\Rightarrow & \ \frac{1}{\E W} \P\biggl[ \left({\min_{1\le j\le
                                k}(E_j-W_j)},\, {\min_{s\in
                                 G}Z_s}\right)\in\cdot, \\
&\hskip 1in     E_1-W_1>0,\ldots, E_k-W_k>0, \,
                                 \min_{j=k+1,\ldots, l}Z_{t_j}>0 \biggr] \notag
\end{align}
weakly in $\bbr^2$.
We conclude both that,  conditionally given $\bigl\{ \min_{t\in
  [a,b]}X_t>u\bigr\}$,  as $u\to\infty$, 
\begin{equation}
\label{t4.new.eq3}\frac{\min_{1\le j\le k}M_{j\vep}-u}{\min_{s\in G}[X_s-u\mu(s)]}\prob0
\end{equation}
and that 
$$
\P\left(\left.\min_{s\in G}[X_s-u\mu(s)]<0\right| \min_{t\in
    [a,b]}X_t>u \right) \to 0\,.
$$
Since $\mu(s)\ge1$ for all $s\in[a,b]$, it follows that
\begin{eqnarray*}
&&\P\left(\left.\min_{1\le j\le k}M_{j\vep}<\min_{s\in G}X_s\right|  
\min_{t\in  [a,b]}X_t>u\right)\\
&\ge&\P\left(\left.\min_{1\le j\le k}M_{j\vep}-u<\min_{s\in
      G}[X_s-u\mu(s)]\right|
\min_{t\in  [a,b]}X_t>u \right)\\
&\ge&\P\left(\left.\left|\frac{\min_{1\le j\le
      k}M_{j\vep}-u}{\min_{s\in G}[X_s-u\mu(s)]}\right|
<1\right| \min_{t\in  [a,b]}X_t>u\right)\\
&&-\P\left(\left.\min_{s\in G}[X_s-u\mu(s)]<0\right|\min_{t\in
   [a,b]}X_t>u \right) \to 1\,.
\end{eqnarray*}
Therefore, 
\eqref{t4.new.eq1} follows. The fact that $\mu(s)>1$ for all $s\in[a,b]\setminus E$ with an appeal to Theorem \ref{t3} implies \eqref{t4.new.claim1}.

In order to prove \eqref{t4.new.claim2}, fix $x_1,\ldots,x_k>0$.  The
same argument as in \eqref{t4.new.eq2} gives us 
\begin{eqnarray*}
&&\lim_{u\to\infty}\P\left( u(M_{1\vep}-u)>x_1,\ldots,u(M_{k\vep}-u)>x_k\Big|
   \min_{t\in [a,b]}X_t >u\right)\\
&=&\frac{1}{\E W}\P\left(E_j-W_j>x_j\,,1\le j\le k, \,  \min_{j=k+1,\ldots, l}Z_{t_j}>0
\right)\\
 &=&\P\left(E_j>x_j\,,1\le j\le k\right)\,,
\end{eqnarray*}
the last equality following by first conditioning on $(Z_t:t\in\bbr)$
and then using the memoryless property of $E_1,\ldots,E_k$. Thus
\eqref{t4.new.claim2} follows. 
\end{proof}

\begin{proof}[Proof of Theorem \ref{t4}]
By \eqref{t4.new.claim2}, \eqref{t4.new.eq3}, and the  fact that
$\mu(s)>1$ for all $s\in[a,b]\setminus E$ we conclude that for $x>0$, 
$$
\P\left( u\bigl( \min_{t\in [a,b]}X_t-u\bigr)>x|\Big| \min_{t\in
    [a,b]}X_t>u\right) \to e^{-(\theta_1+\ldots+\theta_k)x}\,.
$$
By \eqref{t.q1.q2}, the claim of the theorem follows. 
\end{proof}

\begin{proof}[Proof of Theorem \ref{t5}]
By \eqref{t4.new.claim1}, for each $j=1,\ldots, k$, 
\begin{align*}
\lim_{u\to\infty} & \, \P\left(T_{*}=t_j
                    \Bigr|\min_{s\in[a,b]}X_s>u\right) \\
\to & \, \P\bigl( E_j=\min(E_1,\ldots,
E_k)\bigr)=\frac{\theta_j}{\theta_1+\ldots+\theta_k}\,, 
\end{align*}
so the claim of  the theorem will follow once we check that the
measure 
$$
\hat\nu = \sum_{i=1}^k \frac{\theta_i}{\theta_1+\ldots+\theta_k}
\delta_{t_i}
$$
coincides with $\nu_*$. However, by the definition
\eqref{e:deftheta} of the vector $\theta$, the vector
$\Sigma\theta$ has identical positive components. It follows from
Theorem 4.3 (ii) in \ams\ that the measure $\hat\nu$ is optimal
for the minimization problem
$$
\min_{\nu\in M_1\{ t_1,\ldots, t_k\}}\int_{\{ t_1,\ldots,
  t_k\}}\int_{\{ t_1,\ldots, t_k\}}  R(s,t)\nu(ds)\nu(dt). 
$$
The measure $\nu_*$ is also optimal for this problem since it is
optimal for \eqref{p1.eq1}. That is, $\hat\nu$ is optimal for
\eqref{p1.eq1} as well and, since the latter problem has a unique
minimizer, $\hat\nu=\nu_*$. 
\end{proof}

\section*{Acknowledgment} The authors gratefully acknowledge the comments of two anonymous referees, which helped in improving the paper.


\begin{thebibliography}{9}
\expandafter\ifx\csname natexlab\endcsname\relax\def\natexlab#1{#1}\fi

\bibitem[Adler et~al.(2014)Adler, Moldavskaya and
  Samorodnitsky]{adler:moldavskaya:samorodnitsky:2014}
{\sc R.~J. Adler, E.~Moldavskaya {\rm and} G.~Samorodnitsky} (2014): On the
  existence of paths between points in high level excursion sets of {G}aussian
  random fields.
\newblock {\em Annals of Probability\/} 42:1020--1053.

\bibitem[Adler and Taylor(2007)]{adler:taylor:2007}
{\sc R.~J. Adler {\rm and} J.~E. Taylor} (2007): {\em Random Fields and
  Geometry\/}.
\newblock Springer, New York.

\bibitem[Aza\"is and Wschebor(2009)]{azais:wschebor:2009}
{\sc J.~Aza\"is {\rm and} M.~Wschebor} (2009): {\em Level Sets and Extrema of
  Random Processes and Fields\/}.
\newblock Wiley, Hoboken, N.J.

\bibitem[Berman and K{\^o}no(1989)]{berman:kono:1989}
{\sc S.~Berman {\rm and} M.~K{\^o}no} (1989): The maximum of a Gaussian process
  with nonconstant variance: a sharp bound for the distribution tail.
\newblock {\em Annals of Probability\/} 17:632--650.

\bibitem[Dudley(1973)]{dudley:1973}
{\sc R.~Dudley} (1973): Sample functions of the {G}aussian process.
\newblock {\em Annals of Probability\/} 1:66--103.

\bibitem[Guliashvili and Tankov(2016)]{guliashvili:tankov:2016}
{\sc A.~Guliashvili {\rm and} P.~Tankov} (2016): Tail behavior of sums and
  differences of log-normal random variables.
\newblock {\em Econometric reviews\/} 22:444--493.

\bibitem[Piterbarg(1996)]{piterbarg:1996}
{\sc V.~Piterbarg} (1996): {\em Asymptotic Methods in the Theory of Gaussian
  Processes and Fields\/}, volume 148 of {\em Translations of Mathematical
  Monographs\/}.
\newblock American Mathematical Society, Providence, RI.

\bibitem[Samorodnitsky and Shen(2013)]{samorodnitsky:shen:2013}
{\sc G.~Samorodnitsky {\rm and} Y.~Shen} (2013): Is the location of the
  supremum of a stationary process nearly uniformly distributed?
\newblock {\em Annals of Probability\/} 41.

\bibitem[Talagrand(1987)]{talagrand:1987}
{\sc M.~Talagrand} (1987): Regularity of {G}aussian processes.
\newblock {\em Acta Mathematica\/} 159:99--149.

\end{thebibliography}

\end{document}